\documentclass[preprint,12pt]{elsarticle}

\usepackage{amssymb}
\usepackage{amsmath}
\usepackage{amsfonts}
\usepackage{graphicx}
\usepackage{graphics}
\usepackage{tikz}
\usepackage{titlesec}
\usepackage{xspace}

\newtheorem{theorem}{Theorem}

\newtheorem{corollary}[theorem]{Corollary}

\newtheorem{lemma}[theorem]{Lemma}

\newtheorem{remark}[theorem]{Remark}

\def\qed{\vbox{\hrule
 \hbox{\vrule\hbox to 5pt{\vbox to 8pt{\vfil}\hfil}\vrule}\hrule}}

\usepackage{tikz}
\usepackage{tkz-graph}

\journal{.}

\begin{document}

\begin{frontmatter}

\title{On the energy of digraphs}

\author{Juan R. Carmona}

\address{Facultad de Ciencias -- Instituto de ciencias F\'{i}sicas y Matem\'{a}ticas\\
  Universidad Austral de Chile, Independencia 631 -- Valdivia -- Chile. 
 }

\ead{juan.carmona@uach.cl, jcarmona.math@gmail.com}

\begin{abstract}

\noindent 

Let $D$ be a simple digraph with eigenvalues $z_1,z_2,...,z_n$. The energy of $D$ is defined as
$E(D)= \sum_{i=1}^n |Re(z_i)|$, is the real part of the eigenvalue $z_i$.
In this paper a lower bound will be obtained for the spectral radius of $D$,
wich improves some the lower bounds that appear in the literature \cite{G-R}, \cite{T-C}.
This result allows us to obtain an upper bound for the energy of $ D $.
Finally, digraphs are characterized in which this upper bound improves the bounds given in \cite{G-R} and \cite{T-C}.
\end{abstract}

\begin{keyword}
Energy of a digraph, upper bound, spectral radius, lower bound.
\end{keyword}

\end{frontmatter}

\section{Introduction and preliminaries}
Let $D = (V,\Gamma)$ denote a digraph of order $n$, where $V = \{v_1, v_2,..., v_n\}$ 
is the set of vertices of $D$ with $|V| = n$, and  $\Gamma$ is the set of arcs
consisting of ordered pairs of distinct vertices. 
We only consider digraphs with no loops and no miltiple arcs. Two vertices 
$u$ and $v$ of $D$ are called adjacent if they are connected by an arc 
$(u, v) \in \Gamma$, or $(v, u) \in \Gamma$, and doubly adjacent if
$(u, v),(v, u) \in \Gamma$. For any vertex $v_i$, let 
$\overleftrightarrow{N_i}=\{v_j\in V: (v_i,v_j), (v_j,v_i) \in \Gamma\}$ 
denote the vertices doubly adjacents of $v_i$.

A walk $\pi$ of length $l$ from vertex $u$
to vertex $v$ is a sequence of vertices $\pi : u = u_0, u_1,..., u_l = v$, where $(u_{k-1}, u_k)$ is an arc of $D$ for
any $1 \leq k \leq l$. If $u = v$ then $\pi$ is called a closed walk.
A digraph $D$ is strongly connected if for every pair of distinct vertices $ u, v$ of $D$, there exists a walk
from $u$ to $v$ and a walk from $v$ to $u$. A strong component of a digraph $D$ is a maximal subdigraph with
respect to the property of being strongly connected.

Denote the number of closed walks of length $2$
of associated vertex $v_i \in V$ by $c ^{(i)}_2$. Note that $c ^{(i)}_2=\left|\overleftrightarrow{N_i}\right|.$ 
The sequence $(c^{(1)}_2 , c ^{(2)}_2 ,..., c ^{(n)}_2)$ is called closed walk sequence of
length $2$ of $D$. Thus $c_2 =c^{(1)}_2 + c ^{(2)}_2+ \cdots +c ^{(n)}_2$ is the number of all closed walks of length $2$ of $ D$. Denote by $t^{(i)}_2$ the sum of the all closed walks of legth $2$ of the vertices doubly adjacente to $v_i$. Then, $$t^{(i)}_2=\displaystyle \sum_{v_j\in \overleftrightarrow{N_i}}c ^{(j)}_2,$$

 A d\'igraph $D = (V, \Gamma)$ is symmetric if for any $(u, v) \in \Gamma$  also $(v, u) \in \Gamma$, where $u, v \in V$.
A one-to-one correspondence between simple graphs and symmetric digraphs is given by $G\rightarrow \overleftrightarrow{G}$, where
$\overleftrightarrow{G}$ has the same vertex set as the graph G, and each edge $uv$ of $G$ is replaced by a pair of
symmetric arcs $(u, v)$ and $(v, u)$. Under this correspondence, a graph can be identified with a symmetric
digraph.

The  adjacency  matrix $A$ of  the  digraph $D$ is  a  $0–1$  matrix  of  order $n$ with  entries $a_{ij}$,
such  that $a_{ij}=  1$  if $(v_i,v_j) \in  \Gamma$  and $a_{ij}=  0$ otherwise.
Denote the characteristic polynomial of matrix $A$ as $\Phi_{D}(x)$ and 
its eigenvalues  $z_1,z_2,...z_n$, possibly complex (see \cite{C-D-S}, \cite{H-J}.)
The spectral radius of a digraph $D$ is denoted by $\rho=\rho(D)$ 
and defined as $$\displaystyle \rho=\max_{i=1,...,n}\{|z_i|\},$$
where $|z_i|$ denote the modulus of $z_i$.

Ivan Gutman in \cite{G}, introduced the concept of energy for a simple graph $ G $,  as
$$E(G)=\displaystyle \sum_{i=1}^n |\lambda_i|,$$
where $\lambda_1,\lambda_2,...,\lambda_n$ are the eigenvalues of the graph $G$. Details of the
theory of graph energy can be found in the reviews \cite{G1},\cite{G2} and the book \cite{L-S-G}. 
New results of this theory appear in
\cite{A-C-I-R},\cite{Bozkurt},\cite{G-F-Z-G},\cite{J}. It is well known \cite{C} that if $G$ is a graph with $n$ vertices then 
$$E(G)=\displaystyle \frac1\pi \int_{-\infty}^{\infty}\left(n-\frac{ix\Phi_{G}'(ix)}{\Phi_{G}(ix)}\right)dx$$

Pe\~{n}a and Rada in \cite{P-R} showed that for a digraph of $n$ vertices, you have to $$\displaystyle \frac1\pi \int_{-\infty}^{\infty}\left(n-\frac{ix\Phi_{D}'(ix)}{\Phi_{D}(ix)}\right)dx=\sum_{i=1}^{n}|Re(z_i)|,$$ extending the concept of energy for the case of digraphs as
$$E(D)=\sum_{i=1}^n |Re(z_i)|,$$
where $z_1,...,z_n$ are the eigenvalues of $D$ and $Re(z_i)$ 
denotes the real part of $z_i$. For more details about the energy of digraphs,
see \cite{A-B-G},\cite{B},\cite{C-D-S},\cite{C-G-R},\cite{G-R},\cite{L-R},
\cite{P-R},\cite{R},\cite{R2}  and thereferences therein.

In \cite{R}, Rada generalizes the McClelland inequality for any digraph $D$ with $n$ vertices, $a$ arcs
and $c_2$ closed walks of length $2$.
\begin{eqnarray}\label{Mc}
 E(D)\leq \sqrt{\dfrac{n(a+c_2)}{2}},
\end{eqnarray}
with equality in (\ref{Mc}) if and only if $D$ is the direct sum of $\frac n
2$ copies of $\overleftrightarrow{K_2}$.
in order to  obtain an upper bound for energy of a digraph,  Gudi\~{n}o and Rada
in \cite{G-R}, generalizing the idea in \cite{K-M}, showed that the following relation holds:
\begin{eqnarray}\label{1}
 E(D)\leq \rho + \sqrt{(n-1)(a-\rho^{2})}.
\end{eqnarray}
 Then, using the inequality $\dfrac{c_2}{n}\leq \rho$, see \cite{G-R}, they obtained the upper bound
\begin{equation}\label{001}
    E(D) \leq \frac{c_2}{n}+\sqrt{(n-1)\left(a-\left(\frac{c_2}{n}\right)^2\right)}.
\end{equation}
Equality holds in (\ref{001}) if and only if $D$ is
either the empty digraph or $D =\overleftrightarrow{G}$,
where $G$ is either $\frac n2K_2,K_n$, a non-complete 
connected strongly regular graph with two non-trivial 
eigenvalues both with absolute value 
$\sqrt{\frac{\left(a-\left(\frac{c_2}{n}\right)^2\right)}{(n-1)}}$.

Tian and Cui in \cite{T-C}, improve the upper bound (\ref{001}) with the following result
\begin{equation}\label{1}
    E(D) \leq \sqrt{\frac1n \sum_{i=1}^{n}\left(c_2^{(i)}\right)^{2}}+\sqrt{(n-1)\left(a-\frac1n \sum_{i=1}^{n}\left(c_2^{(i)}\right)^{2}\right)}.
\end{equation}
The equality in (\ref{1}) holds if and only if 
$D =\overleftrightarrow{G}$, where $G$ is either
$\frac n2K_2,K_n$, a non-complete connected strongly
regular graph with two non-trivial eigenvalues both
with absolute value 
$\sqrt{\frac{\left(a-\frac{\sum_{i=1}^{n}\left(c_2^{(i)}\right)^{2}}{n} \right)}{(n-1)}},$ or $nK_1.$

In this work, motivated by the strategies used in \cite{Y-L-T} and \cite{T-C}, we get a lower bound for the spectral radius $\rho$ in terms of $n$, $c_2^{(i)}$ and $t_2^{(i)}$, improving the known results. Using this result we obtain an upper bound of $E(G)$ in terms of $n$,$a$, $c_2^{(i)}$ and $t_2^{(i)}$. In addition, we will show that this bound improving and generalize the bounds given for graphs and digraphs in \cite{G-R}, \cite{T-C} and \cite{Y-L-T}.

 \section{Improving lower bound on the spectral radius of a digraph}
Obtaining lower bounds for the spectral radius $\rho$ of a digraph $D$  is essential to obtain new bounds for $ E(D) $, see section \ref{s3}.
\begin{remark}\label{REM}
Recall that for an $n$-by-$n$ matrix $A = (a_{ij})$, its geometric symmetrization, denoted by $S(A) = (s_{ij})$, is
the matrix with entries $s_{ij} = \sqrt{a_{ij}a_{ji}}$ for any $i, j = 1, 2,..., n$.Thus, it is holds that
\begin{itemize}
    \item [i.-] $c^{(i)}_2 =\displaystyle \sum_{j=1}^n s_{ij}$ for any vertex
$v_i \in V$. 
    \item [ii.-]  $\displaystyle \sum_{i=i}^{n}t^{(i)}_2=\sum_{i=1}^{n}\left(c_2^{(i)}\right)^{2}$.
    
    \item [iii.-] $\rho(A)\geq \rho(S(A))=\sqrt{\rho(S(A^2))}$.
\end{itemize}
\end{remark}

In \cite{G-R} obtained the following theorem:
\begin{theorem}[Gudi\~{n}o and Rada]
Let $D$ be a digraph with $n$ vertices and $c_2$ closed walks of length $2$. Then
\begin{equation}\label{c-i}
    \rho(D)\geq \dfrac{c_2}{n}
\end{equation}
Equality holds if and only if
$$D = \overleftrightarrow{G} +\{\text{possibly some arcs that do not belong to cycles}\},$$
where $G$ is a $\frac{c_2}{n}$-regular graph.
\end{theorem}

In \cite{T-C}, a better lower bound is presented
\begin{theorem}[Tian and Cui]\label{T-C}
Let $D$ digraph with $n$ vertices.
Also let $c_2^{(1)},c_2^{(2)},...,c_2^{(n)}$
be the closed walk sequence of length $2$ of $D$. Then 
\begin{equation}\label{-cotaindice}
    \rho(D)\geq \sqrt{\frac{\sum_{i=1}^{n}\left(c_2^{(i)}\right)^{2}}{n}},
\end{equation}
whit equality in (\ref{-cotaindice}) if only if
$$D =\overleftrightarrow{G}+ \{ \textit{possibly some arcs that do not belong to cycles}\},$$
where each connected component of $G$ is either an $r$-regular
graph or an $(r_1, r_2)$-semiregular bipartite graph,
satisfying $r_1r_2 =\frac{\sum_{i=1}^{n}\left(c_2^{(i)}\right)^{2}}{n}$.
\end{theorem}

Now, we give the following lemma which is important for finding improved lower bound
for the spectral radius of D.
\begin{lemma}{\cite{R}}\label{LEMA}
Let $D$ be a digraph with $n$ vertices, $a$ arcs and $c_2$ closed walks of length $2$. If $z_1, z_2,...,z_n$
are the eigenvalues of $D$, then
\begin{itemize}
    \item[i.-]$\displaystyle \sum_{i=1}^n\left(Re(z_i)\right)^{2}-\sum_{i=1}^n\left(Im(z_i)\right)^{2}=c_2;$
    \item[ii.-]  $\displaystyle \sum_{i=1}^n\left(Re(z_i)\right)^{2}+\sum_{i=1}^n\left(Im(z_i)\right)^{2}\leq a.$
\end{itemize}
\end{lemma}

A first main result in this work is the following
\begin{theorem}\label{teo}
Let $D$ digraph with $n$ vertices, with sequences $c_2^{(1)},c_2^{(2)},...,c_2^{(n)}$ and $t^{(1)}_2,t^{(2)}_2,...,t^{(n)}_2$. Then 
\begin{equation}\label{cotaindice}
    \rho(D)\geq \sqrt{\frac{\sum_{i=1}^{n}\left(t_2^{(i)}\right)^{2}}{\sum_{i=1}^{n}\left(c_2^{(i)}\right)^{2}}},
\end{equation}
whit equality in (\ref{cotaindice}) if only if 
$$D =\overleftrightarrow{G}+ \{ \textit{possibly some arcs that do not belong to cycles}\},$$ 
where each connected component of $G$ is either an $r$-regular graph or an 
$(r_1, r_2)$-semiregular bipartite graph, satisfying 
$r_1r_2 =\frac{\sum_{i=1}^{n}\left(t_2^{(i)}\right)^{2}}{\sum_{i=1}^{n}\left(c_2^{(i)}\right)^{2}}$.
\end{theorem}
\begin{proof}
  for the Rayleigh quotient
\begin{eqnarray}
 \sqrt{\rho(S(A)^2)}=\sqrt{max_{x\neq 0}\frac{x^TS(A)^2x }{x^Tx}}\geq \sqrt{\frac{c^TS(A)^2c }{c^Tc}}=\sqrt{\frac{\sum_{i=1}^{n}\left(t_2^{(i)}\right)^{2}}{\sum_{i=1}^{n}\left(c_2^{(i)}\right)^{2}}},
\end{eqnarray}
where $c=(c^{(1)}_2 , c ^{(2)}_2 ,..., c ^{(n)}_2)^T$. Thus, we obtain (\ref{cotaindice}).
To prove equality, we will use the ideas used in the works \cite{T-C}, \cite{H-S}, \cite{G-R},\cite{B-B-Z}. Indeed, suppose now that the equality in (\ref{cotaindice}) holds, then $$\rho(S(A)^2)=\frac{c^TS(A)^2c}{c^Tc},$$ then $c$ is a positive eigenvector of $S(A)^2$
corresponding to the eigenvalue $\rho(S(A)^2),$ 
either one or two. Next we consider three cases.
\begin{itemize}
    \item [Case 1:] D is strongly connected. \\
     A is a irreducible matrix in this case. If $A > S(A)$, then
$\rho(A) > \rho(S(A))$ as A is irreducible (see \cite{B-P}, Corollary 2.1.5), 
this contradicts our assumption of equality. Therefore we have that $A$
is a symmetric matrix, which implies that $D =\overleftrightarrow{G}$.
In this case $G$ is a connected simple graph. 
Then, similar to the proofs in (\cite{H-S}, Theorem 3.1),
one can easily obtain that $G$ is either an $r$-regular graph
or $(r_1, r_2)$-semiregular bipartite graph, satisfying $r^ 2 = r_1 r_2 =\frac{\sum_{i=1}^{n}\left(t_2^{(i)}\right)^{2}}{\sum_{i=1}^{n}\left(c_2^{(i)}\right)^{2}}$.

\item [Case 2:] $D$ is direct sum of its disjoint strongly connected components
$D_1, D_2, . . . , D_s$.
\noindent
Let $A_k$ be the $n_k$-by-$n_k$ adjacency matrix of $D_k$ and $\sum_{k=1}^{s}n_k = n$. In this case 
$$
A^{2}=\left( \begin{array}{cccc}
    A_1^{2} & & &  \\
     & A_2^{2} & &  \\
     & & \ddots&  \\
     & & &  A_s^{2}
\end{array}\right),$$
where the rest of the unspecified entries are $0$. Since the equality holds in (\ref{cotaindice}), we have
$$
\begin{array}{lllll}
      \sqrt{\rho(S(A)^2)} &= & \displaystyle \sqrt{ \max_{x\neq 0}\frac{x^TS(A)^2x }{x^Tx}}= \sqrt{\displaystyle\frac{c^TS(A)^2c }{c^Tc}}\\
     &=  & \displaystyle \sqrt{\sum_{k=1}^s\frac{c_{n_k}^TS(A_k)^2c_{n_k} }{n_k}\frac{n_k}{c_{n_k}^Tc_{n_k}}}  \\
     &\leq & \displaystyle \sqrt{\sum_{k=1}^s\frac{{n_k}\rho\left(S(A_k)^2\right)}{n}} \leq \sqrt{\max_{k}\rho(S(A_{k})^2)}\\
     &=& \displaystyle \max_{k} \sqrt{\rho(S(A_{k})^2)}=\sqrt{\rho(S(A)^2)}\\
     &=& \sqrt{\rho(A)^2}
\end{array}
$$

which implies that, for every $k = 1, 2, . . . , s$, 
$$\rho(A)=\sqrt{\rho(A^2)}=\sqrt{\rho(A_k^2)}=\sqrt{\rho(S(A_k^2))}=\sqrt{\sum_{k=1}^s\frac{c_{n_k}^TS(A_k^2)c_{n_k} }{n_k}}$$
Then, from Case (1) each $D_k =\overleftrightarrow{G_k}$, where each connected component $G_k$ is either an
$r$-regular graph or $(r_1, r_2)$-semiregular bipartite graph, satisfying $r^ 2 = r_1r_2 = \frac{\sum_{i=1}^{n}\left(t_2^{(i)}\right)^{2}}{\sum_{i=1}^{n}\left(c_2^{(i)}\right)^{2}}.$

\item [Case 3:] $\widehat{D}$ is a digraph obtained from $D$ by deleting those arcs of $D$
that do not belong to any cycle.\\
Then $S(A) = S(A(\widehat{D}))$, where $A(\widehat{D})$ is the adjacency matrix of $\widehat{D}$.
Clearly, $D$ and $\widehat{D}$ have the same cycle structure.
By Theorem 1.2 in \cite{C-D-S}, we have that $\Phi_{D}(x)=\Phi_{\widehat{D}}(x)$,
which implies that $D$ and $\widehat{D}$ also have the same eigenvalues.
On the other hand, since $\widehat{D}$ is direct sum of its some disjoint
strongly connected components, then Case (2) implies that 
$\widehat{D} =\overleftrightarrow{G}$ and each connected component of $G$
is either an $r$-regular graph or an $(r_1, r_2)$-semiregular bipartite graph,
satisfying 
$r^ 2 = r_1 r_2 =\frac{\sum_{i=1}^{n}\left(t_2^{(i)}\right)^{2}}{\sum_{i=1}^{n}\left(c_2^{(i)}\right)^{2}}$.
Hence,$D =\overleftrightarrow{G}+ \{ \textit{possibly some arcs that do not belong to cycles}\}.$

\end{itemize}
Conversely, suppose that
$D =\overleftrightarrow{G}+ \{ \textit{possibly some arcs that do not belong to cycles}\},$
where each connected component of $G$ is either an $r$-regular graph or an $(r_1, r_2)$-semiregular bipartite
graph, satisfying $r^ 2 = r_1r_2 = \frac{\sum_{i=1}^{n}\left(t_2^{(i)}\right)^{2}}{\sum_{i=1}^{n}\left(c_2^{(i)}\right)^{2}}.$ It is easy to check that the equality in (\ref{cotaindice}) holds.
\end{proof}

The result given in \cite{Y-L-T} is here re-obtained considering $D =\overleftrightarrow{G}$. 

\begin{corollary}
Let $G$ be a nonempty graph with degree sequence $d_1, d_2, · · · , d_n$ and $2$-degree
sequence $t_1, t_2, · · · , t_n$. Then
$$\lambda_1(G) \geq
\sqrt{\frac{\sum_{i=1}^{n} t_i^{2}}{\sum_{i=1}^{n} d_i^{2}}},$$
with equality if and only if G is a pseudo-regular graph or a pseudo-semiregular bipartite graph.
\end{corollary}
The following remark allows us to prove that the bound given in (\ref{cotaindice}) is better than the bound (\ref{-cotaindice}) given in \cite{T-C} and consequently better than the bound given in \cite{G-R}. 
\begin{remark}\label{Re}
 Note that 
$$\sqrt{\frac1n \sum_{i=1}^{n}\left(c_2^{(i)}\right)^{2}} \leq\sqrt{\frac{\sum_{i=1}^{n}\left(t_2^{(i)}\right)^{2}}{\sum_{i=1}^{n}\left(c_2^{(i)}\right)^{2}}}.$$
\begin{proof}
    In efect, by Cauchy-Schwarz inequality, we have
    $$\left(\sum_{i=1}^{n}t_2^{(i)}\right)^{2}\leq n\sum_{i=1}^{n}\left(t_2^{(i)}\right)^{2}.$$
Using this inequality and Remark \ref{REM} (part ii),
$$\sqrt{\frac{\sum_{i=1}^{n}\left(t_2^{(i)}\right)^{2}}{\sum_{i=1}^{n}\left(c_2^{(i)}\right)^{2}}} \geq \sqrt{\frac{\left(\sum_{i=1}^{n}t_2^{(i)}\right)^{2}}{n \sum_{i=1}^{n}\left(c_2^{(i)}\right)^{2}}}=\sqrt{\frac1n \sum_{i=1}^{n}\left(c_2^{(i)}\right)^{2}}.$$
\end{proof}
\end{remark}
\section{An upper bound for the energy of a digraph}\label{s3}
In this section, using the strategies given in articles \cite{B-B-Z},\cite{T-C} and \cite{Y-L-T}, we will construct a lower bound for the energy of digraph $D$, using the result obtained in \ref{teo}.

\begin{theorem}
Let $D$ digraph with $n$ vertices, $a$ arcs, with sequences $c_2^{(1)},c_2^{(2)},...,c_2^{(n)}$ and $t^{(1)}_2,t^{(2)}_2,...,t^{(n)}_2$. Then 

\begin{eqnarray}\label{t1}
 E(D) \leq \sqrt{\frac{\sum_{i=1}^{n}\left(t^{(i)}_2\right)^{2}}{ \sum_{i=1}^{n}\left(c_2^{(i)}\right)^{2}}}+\sqrt{(n-1)\left(a-\frac{\sum_{i=1}^{n}\left(t^{(i)}_2\right)^{2}}{ \sum_{i=1}^{n}\left(c_2^{(i)}\right)^{2}}\right)}.
\end{eqnarray}
The equality in (\ref{t1}) holds if and only if $D =\overleftrightarrow{G}$, where $G$ is either $\frac n2K_2,K_n$, a non-complete connected strongly regular graph with two non-trivial eigenvalues both with absolute value $\sqrt{\frac{\left(a-\frac{\sum_{i=1}^{n}\left(t^{(i)}_2\right)^{2}}{ \sum_{i=1}^{n}\left(c_2^{(i)}\right)^{2}} \right)}{(n-1)}},$ or $nK_1.$
\end{theorem}

\begin{proof}
Let $\rho=z_1, z_2, . . . , z_n$ be the eigenvalues of the digraph $D$ such that $Re(z_1) \geq Re(z_2) \geq \cdots\geq Re(z_n)$. By Lemma \ref{LEMA} (part $(ii)$),  we have
\begin{equation}\label{arcs}
    \sum_{i=2}^n\left(Re(z_i)\right)^{2} \leq a- \rho^{2},
\end{equation}
where $a$ is the number of arcs. Using (\ref{arcs}) together with the Cauchy-Schwartz inequality, we obtain the
inequality
 $$\sum_{i=2}^n|Re(z_i)|\leq \sqrt{(n-1) \sum_{i=2}^n\left(Re(z_i)\right)^{2}}\leq \sqrt{(n-1)(a-\rho^{2})}.$$
Thus, we must have
\begin{eqnarray}\label{En}
 E(D)\leq \rho + \sqrt{(n-1)(a-\rho^{2})}.
\end{eqnarray}
Now, consider the function $f(x)=x+ \sqrt{(n-1)(a-x^{2})}, \, \, x\in [0,\sqrt{a}]$. It is easy to see that the function $f(x)$ increases strictly on the interval $\left[0, \sqrt{\frac an}\right]$ and decreases strictly on $\left[\sqrt{\frac an},\sqrt{a}\right].$ At this point, we have to analyze two cases:
\begin{itemize}
    \item [Case 1.] $a \leq n \displaystyle \frac{\sum_{i=1}^{n}\left(t_2^{(i)}\right)^{2}}{\sum_{i=1}^{n}\left(c_2^{(i)}\right)^{2}}$\\
    Then by Theorem \ref{teo} and inequality (\ref{arcs}), we have 
   $$\displaystyle \sqrt{\frac{a}{n}}\leq  \sqrt{\frac{\sum_{i=1}^{n}\left(t_2^{(i)}\right)^{2}}{\sum_{i=1}^{n}\left(c_2^{(i)}\right)^{2}}}\leq \rho \leq \sqrt{a}.$$ Thus $f (\rho) \leq f\left( \sqrt{\frac{\sum_{i=1}^{n}\left(t_2^{(i)}\right)^{2}}{\sum_{i=1}^{n}\left(c_2^{(i)}\right)^{2}}}\right)$, because $f$ is decreasing in  $\left[\sqrt{\frac an},\sqrt{a}\right].$ This implies that the inequality (\ref{t1}) holds. On the other hand, if the equality in (\ref{t1}) holds, then $$\rho = \sqrt{\frac{\sum_{i=1}^{n}\left(t_2^{(i)}\right)^{2}}{\sum_{i=1}^{n}\left(c_2^{(i)}\right)^{2}}},$$ later by theorem \ref{teo}, we have that $$D =\overleftrightarrow{G}+ \{ \textit{possibly some arcs that do not belong to cycles}\},$$ where each connected component of $G$ is either an $r$-regular graph or an $(r_1, r_2)$-semiregular bipartite graph, satisfying $r_1r_2 =\frac{\sum_{i=1}^{n}\left(t_2^{(i)}\right)^{2}}{\sum_{i=1}^{n}\left(c_2^{(i)}\right)^{2}}$. Noting that $c_2 \leq a$. By Theorem 2.1 in \cite{Y-L-T1}, we obtain
$$
\begin{array}{cccc}
     E(D)&= &E(G)\leq \displaystyle \sqrt{\frac{\sum_{i=1}^{n}\left(t^{(i)}_2\right)^{2}}{ \sum_{i=1}^{n}\left(c_2^{(i)}\right)^{2}}}+\sqrt{(n-1)\left(c_2-\frac{\sum_{i=1}^{n}\left(t^{(i)}_2\right)^{2}}{ \sum_{i=1}^{n}\left(c_2^{(i)}\right)^{2}}\right)} \\
     & \leq & \displaystyle \sqrt{\frac{\sum_{i=1}^{n}\left(t^{(i)}_2\right)^{2}}{ \sum_{i=1}^{n}\left(c_2^{(i)}\right)^{2}}}+\sqrt{(n-1)\left(a-\frac{\sum_{i=1}^{n}\left(t^{(i)}_2\right)^{2}}{ \sum_{i=1}^{n}\left(c_2^{(i)}\right)^{2}}\right)}=E(D),
\end{array}
$$
 which implies $c_2 = a$, this way we have to
 $$E(D)=\displaystyle \sqrt{\frac{\sum_{i=1}^{n}\left(t^{(i)}_2\right)^{2}}{ \sum_{i=1}^{n}\left(c_2^{(i)}\right)^{2}}}+\sqrt{(n-1)\left(c_2-\frac{\sum_{i=1}^{n}\left(t^{(i)}_2\right)^{2}}{ \sum_{i=1}^{n}\left(c_2^{(i)}\right)^{2}}\right)}.$$
Using the Theorem 2.1 in \cite{Y-L-T1}, we obtain the conditions of equality.
    
    \item[Case 2:] $a >n \displaystyle \frac{\sum_{i=1}^{n}\left(t_2^{(i)}\right)^{2}}{\sum_{i=1}^{n}\left(c_2^{(i)}\right)^{2}}.$ \\
  Using Remark \ref{Re}, we have
  $$0\leq \sqrt{\frac1n \sum_{i=1}^{n}\left(c_2^{(i)}\right)^{2}} \leq\sqrt{\frac{\sum_{i=1}^{n}\left(t_2^{(i)}\right)^{2}}{\sum_{i=1}^{n}\left(c_2^{(i)}\right)^{2}}} \leq \sqrt{\frac{a}{n}}.$$
  Therefore, we have to $\displaystyle f\left( \sqrt{\frac1n \sum_{i=1}^{n}\left(c_2^{(i)}\right)^{2}}\right) \leq f\left(\sqrt{\frac{\sum_{i=1}^{n}\left(t_2^{(i)}\right)^{2}}{\sum_{i=1}^{n}\left(c_2^{(i)}\right)^{2}}} \right),$ because, $f$ is increasing in $\left[0, \sqrt{\frac an}\right]$. Then by Theorem 2 in \cite{B-B-Z} the inequality (\ref{t1}) holds.\\
  Assume now that equality holds in (\ref{t1}), then we have that $$E(D)=\displaystyle f\left( \sqrt{\frac1n \sum_{i=1}^{n}\left(c_2^{(i)}\right)^{2}}\right) = f\left(\sqrt{\frac{\sum_{i=1}^{n}\left(t_2^{(i)}\right)^{2}}{\sum_{i=1}^{n}\left(c_2^{(i)}\right)^{2}}} \right),$$
  Thus, the characteristics of $ D $ are obtained from the conditions of the Theorem 2 in \cite{B-B-Z}
\end{itemize}
\end{proof}
The  result  given  in  \cite{Y-L-T1}  is  here re-obtained considering $D =\overleftrightarrow{G}$.
\begin{corollary}
Let $G$ be a nonempty graph with $n$ vertices, $m$ edges, degree sequence $d_1, d_2, · · · , d_n$ and $2$-degree
sequence $t_1, t_2, · · · , t_n$. Then
$$E(G) \leq
\sqrt{\frac{\sum_{i=1}^{n} t_i^{2}}{\sum_{i=1}^{n} d_i^{2}}}+\sqrt{(n-1)\left(2m-\frac{\sum_{i=1}^{n} t_i^{2}}{\sum_{i=1}^{n} d_i^{2}}\right)}.$$
Equality holds if and only if one of the following statements holds:
\begin{itemize}
    \item[(1)] $G\cong \frac n2 K_2;$
    \item [(2)] $G\cong K_n;$
     \item [(3)] $G$ is a non-bipartite connected $p$ seudo-regular graph with three distinct eigenvalues $\left(p, \sqrt{\frac{2m-p^2}{n-1}},-\sqrt{\frac{2m-p^2}{n-1}} \right),$ where $p>\sqrt{\frac{m}{n}}.$
\end{itemize}
\end{corollary}

\begin{remark}
 Consider the collection of digraphs of $ n $ vertices, $a$ arcs and $c_2$ walks of length  $2$
denoted and defined by:
 $$\Gamma=\{D: an< (c_2)^2 \}.$$
 Si $D \in \Gamma$, then
 $$\sqrt{\frac a n}\leq \frac{c_2}{n}\leq \sqrt{\frac1n \sum_{i=1}^{n}\left(c_2^{(i)}\right)^{2}}\leq \sqrt{\frac{\sum_{i=1}^{n}\left(t_2^{(i)}\right)^{2}}{\sum_{i=1}^{n}\left(c_2^{(i)}\right)^{2}}} \leq \rho \leq \sqrt{a}.$$
Since the function $f$ is strictly decreasing on the interval  $\left[\sqrt{\frac an},\sqrt{a}\right],$ we have that:
$$E(D) \leq f(\rho) \leq f\left(\sqrt{\frac{\sum_{i=1}^{n}\left(t_2^{(i)}\right)^{2}}{\sum_{i=1}^{n}\left(c_2^{(i)}\right)^{2}}}\right)\leq f\left(\sqrt{\frac1n \sum_{i=1}^{n}\left(c_2^{(i)}\right)^{2}} \right)\leq f\left( \frac{c_2}{n} \right)$$

In this way, we can affirm that for all $G \in \Gamma$,   the bound given in (\ref{t1}) is better than the bound (\ref{1}) given in \cite{T-C} and consequently better than the bound (\ref{001}) given in \cite{G-R}.
\end{remark}

\end{document}